\documentclass[10pt]{article}
\usepackage{amsfonts,amssymb,amsmath,amsthm}
\date{}

\theoremstyle{plain}
\newtheorem{theorem}{Theorem}[section]
\newtheorem{proposition}{Proposition}[section]
\newtheorem{lemma}{Lemma}[section]
\newtheorem{corollary}{Corollary}[section]
\theoremstyle{definition}
\newtheorem{definition}{Definition}[section]
\newtheorem{remark}{Remark}[section]
\numberwithin{equation}{section}
\numberwithin{theorem}{section}
\numberwithin{proposition}{section}
\numberwithin{lemma}{section}
\numberwithin{corollary}{section}
\numberwithin{definition}{section}

\newcommand{\R}{\mathbb{R}}
\newcommand{\N}{\mathbb{N}}
\newcommand{\Z}{\mathbb{Z}}
\newcommand{\D}{\mathcal{D}}

\newcommand{\esslim}{\operatornamewithlimits{ess\,lim}}

\newcommand{\essinf}{\operatornamewithlimits{ess\,inf}}
\newcommand{\esssup}{\operatornamewithlimits{ess\,sup}}
\newcommand{\sign}{\operatorname{sign}}

\newcommand{\<}{\langle}
\renewcommand{\>}{\rangle}

\newcommand{\const}{\mathrm{const}}
\renewcommand{\div}{\operatorname{div}}
\newcommand{\supp}{\operatorname{supp}}
\title{To the theory of entropy sub-solutions of degenerate non-linear parabolic equations}
\author{Evgeny Yu. Panov\footnote{Novgorod State University, 41, B.St-Petersburgskaia str., 173003 Veliky Novgorod, Russia}}
\begin{document}
\maketitle
\begin{abstract}
We prove existence of the largest entropy sub-solution and the smallest entropy super-solution to the Cauchy problem for a nonlinear degenerate parabolic equation with only continuous flux and diffusion functions. Applying this result, we establish the uniqueness of entropy solution with periodic initial data. The more general comparison principle is also proved in the case when at least one of the initial functions is periodic.

\end{abstract}

\maketitle

\section{Introduction}\label{sec1}
In the half space $\Pi=\R_+\times\R^n$, $\R_+=(0,+\infty)$, we consider a nonlinear parabolic equation
\begin{equation}\label{1}
u_t+\div_x\varphi(u)-\Delta_x g(u)=0,
\end{equation}
where the flux vector $\varphi(u)=(\varphi_1(u),\ldots,\varphi_n(u))$ and the diffusion function $g(u)$ are merely continuous: $\varphi_i(u)\in C(\R)$, $i=1,\ldots,n$, $g(u)\in C(\R)$ and $g(u)$ is nonstrictly increasing. Since $g(u)$ may be constant on nontrivial intervals, (\ref{1}) is a degenerate
(hyperbolic-parabolic) equation. In particular case $g\equiv\const$ it reduces to a first order conservation law
 \begin{equation}\label{con}
u_t+\div_x \varphi(u)=0.
\end{equation}
Equation (\ref{1}) is endowed with the initial condition
\begin{equation}\label{2}
u(0,x)=u_0(x).
\end{equation}
We recall the notion of entropy solution (as well as entropy sub- and super-solution) in the sense of Carrillo \cite{Car}. We denote $v^+=\max(v,0)$. Let $H(v)=\sign^+ v=\left\{\begin{array}{lr} 1, & v>0, \\ 0, & v\le 0 \end{array}\right.$ be the Heaviside function.

\begin{definition}\label{def1} A function $u=u(t,x)\in L^\infty(\Pi)$ is called an entropy sub-solution (e.sub-s. for short) to problem (\ref{1}), (\ref{2}) if the generalized gradient $\nabla_x g(u)\in L^2_{loc}(\Pi,\R^n)$, for each $k\in\R$
\begin{equation}\label{esub}
((u-k)^+)_t+\div_x[H(u-k)(\varphi(u)-\varphi(k))]-\Delta_x((g(u)-g(k))^+)\le 0
\end{equation}
in the sense of distributions on $\Pi$ (in $\D'(\Pi)$), and
\begin{equation}\label{isub}
\esslim_{t\to 0+} (u(t,x)-u_0(x))^+=0 \ \mbox{ in } L^1_{loc}(\R^n);
\end{equation}
A function $u=u(t,x)\in L^\infty(\Pi)$ is called an entropy super-solution (e.super-s.) to problem (\ref{1}), (\ref{2}) if $\nabla_x g(u)\in L^2_{loc}(\Pi,\R^n)$, for each $k\in\R$
\begin{equation}\label{esuper}
((k-u)^+)_t+\div_x[H(k-u)(\varphi(k)-\varphi(u))]-\Delta_x((g(k)-g(u))^+)\le 0 \ \mbox{ in } \D'(\Pi),
\end{equation}
and
\begin{equation}\label{isuper}
\esslim_{t\to 0+} (u_0(x)-u(t,x))^+=0 \ \mbox{ in } L^1_{loc}(\R^n);
\end{equation}
Finally, a function $u=u(t,x)\in L^\infty(\Pi)$ is called an entropy solution (an e.s.) to problem (\ref{1}), (\ref{2})
if this function is an e.sub-s. and an e.super-s. of this problem simultaneously.
\end{definition}
Entropy condition (\ref{esub}) means that for each test function $f=f(t,x)\in C_0^\infty(\Pi)$, $f\ge 0$,
\begin{align}\label{esubi}
\int_\Pi H(u-k)\{(u-k)f_t+[\varphi(u)-\varphi(k)-\nabla_x g(u)]\cdot\nabla_x f\}dtdx=\nonumber\\
\int_\Pi \{(u-k)^+f_t+H(u-k)(\varphi(u)-\varphi(k))\cdot\nabla_xf+(g(u)-g(k))^+\Delta_x f\}dtdx\ge 0
\end{align}
(here and in the sequel we denote by $\cdot$ the scalar multiplication of finite-dimensional vectors). Similarly we understand entropy relation (\ref{esuper}).

In the case of conservation laws (\ref{con}) the notion of e.s. of (\ref{con}), (\ref{2}) coincides with the known notion of generalized entropy solution in the sense of Kruzhkov \cite{Kr}. It is known that e.s. of (\ref{1}), (\ref{2}) always exists but in the multidimensional case $n>1$ may be nonunique. For conservation laws (\ref{con}) the corresponding examples were constructed in \cite{KrPa1,KrPa2}. In the case $\varphi(u)\in C^1(\R)$ the uniqueness is well-known. Some sufficient conditions for uniqueness of e.s., which extends results of \cite{KrPa2}, are contained in \cite{AndMal}.

\begin{remark}\label{rem1}

(i) As directly follows from the definition, a function $u=u(t,x)$ is an e.super-s. of (\ref{1}), (\ref{2}) if and only if the function $-u$ is an e.sub-s. to the problem
\begin{equation}\label{red}
u_t-\div_x\varphi(-u)-\Delta (-g(-u))=0, \quad u(0,x)=-u_0(x).
\end{equation}

(ii) Taking in (\ref{esub}) $k=-\|u\|_\infty$, we obtain that an e.sub-s. $u=u(t,x)$ satisfies the relation
\begin{equation}\label{wsub}
u_t+\div_x\varphi(u)-\Delta_x g(u)\le 0 \ \mbox{ in } \D'(\Pi).
\end{equation}
Similarly, taking in (\ref{esuper}) $k=\|u\|_\infty$, we arrive at
\begin{equation}\label{wsuper}
u_t+\div_x\varphi(u)-\Delta_x g(u)\ge 0 \ \mbox{ in } \D'(\Pi).
\end{equation}
As follows from (\ref{wsub}) and (\ref{wsuper}), any e.s. satisfies equation (\ref{1}) in $\D'(\Pi)$, i.e., it is a weak solution of this equation.

It is natural to call a function $u=u(t,x)\in L^\infty(\Pi)$ a \textit{weak sub-solution} (weak sub-s.), respectively -- a \textit{weak super-solution} (weak super-s.) of problem (\ref{1}), (\ref{2}) if $\nabla_x g(u)\in L^2_{loc}(\Pi,\R^n)$, and $u$ satisfies (\ref{wsub}), (\ref{isub}), respectively -- (\ref{wsuper}), (\ref{isuper}).

(iii) Since for every function $p(u)$ and $k\in\R$
$$
p(\max(u,k))=p(k)+H(u-k)(p(u)-p(k)), \quad p(\min(u,k))=p(k)-H(k-u)(p(k)-p(u)),
$$
we can rewrite entropy relations (\ref{esub}), (\ref{esuper}) in the equivalent forms
\begin{align}\label{esub1}
(\max(u,k))_t+\div_x\varphi(\max(u,k))-\Delta_x g(\max(u,k))\le 0 \ \mbox{ in } \D'(\Pi), \\
\label{esuper1}
(\min(u,k))_t+\div_x\varphi(\min(u,k))-\Delta_x g(\min(u,k))\ge 0 \ \mbox{ in } \D'(\Pi),
\end{align}
respectively.
\end{remark}

The main results of the paper are contained in the following two theorems.

\begin{theorem}\label{th2}
There exist the unique largest e.sub-s. $u_+(t,x)$ and the smallest e.super-s. $u_-(t,x)$ of the problem (\ref{1}), (\ref{2}). Besides, $u_-(t,x)\le u_+(t,x)$.
\end{theorem}

On the base of this result we establish the following \textit{comparison principle}.

\begin{theorem}\label{th4}
Let functions $u=u(t,x)$, $v=v(t,x)$ be an e.sub-s. and an e.super-s. of (\ref{1}), (\ref{2}) with corresponding initial data $u_0(x)$, $v_0(x)$, and $u_0(x)\le v_0(x)$. If at least one of the initial functions is periodic then $u(t,x)\le v(t,x)$ a.e. in $\Pi$.
\end{theorem}

It is clear that this comparison principle implies uniqueness of e.s. of the problem  (\ref{1}), (\ref{2}) with periodic initial data.

\section{Preliminaries}
It is useful to formulate the notion of e.sub-s. of (\ref{1}), (\ref{2}) in the form of a single integral inequality.

\begin{proposition}\label{pro1}
A function $u=u(t,x)\in L^\infty(\Pi)$ such that $\nabla_xg(u)\in L^2_{loc}(\Pi,\R^n)$ is an e.s. of (\ref{1}), (\ref{2})
if and only if for every $k\in\R$ and each nonnegative test function $f=f(t,x)\in C_0^\infty(\bar\Pi)$, where $\bar\Pi=[0,+\infty)\times\R^n$,
\begin{equation}\label{eint}
\int_{\Pi}H(u-k)[(u-k)f_t+(\varphi(u)-\varphi(k))\cdot\nabla_x f+(g(u)-g(k))\Delta_x f]dtdx+\int_{\R^n} (u_0(x)-k)^+f(0,x)dx\ge 0.
\end{equation}
\end{proposition}
\begin{proof}
Let $E$ be a set of $t>0$ such that $(t,x)$ is a Lebesgue point of $u(t,x)$ for almost all $x\in\R^n$. It is rather well-known (see for example \cite[Lemma~1.2]{PaJHDE}) that $E$ is a set of full measure and $t\in E$ is a common Lebesgue point of the functions $t\to\int_{\R^n} u(t,x)b(x)dx$ for all $b(x)\in L^1(\R^n)$. Since every Lebesgue point of $u$ is also a Lebsesgue point of $p(u)$ for an arbitrary function $p\in C(\R)$, we may replace $u$ in the above property by $p(u)$, and in particular by $(u-k)^+$, $k\in\R$. We choose a function $\omega(s)\in C_0^\infty(\R)$, such that $\omega(s)\ge 0$, $\supp\omega\subset [0,1]$, $\int\omega(s)ds=1$, and define the sequences $\omega_r(s)=r\omega(rs)$, $\theta_r(s)=\int_{-\infty}^s\omega_r(\sigma)d\sigma=\int_{-\infty}^{rs}\omega(\sigma)d\sigma$, $r\in\N$. Obviously, the sequence $\omega_r(s)$ converges as $r\to\infty$ to the Dirac $\delta$-measure weakly in $\D'(\R)$ while the sequence $\theta_r(s)$ converges to the Heaviside function $H(s)$ pointwise and in $L^1_{loc}(\R)$. Notice that $0\le\theta_r(s)\le 1$.
We take $f=f(t,x)\in C_0^\infty(\bar\Pi)$, $f\ge 0$, and $t_0\in E$. Applying (\ref{esub}) to the nonnegative test function $\theta_r(t-t_0)f(t,x)\in C_0^\infty(\Pi)$, we arrive at the relation
\begin{align}\label{i1}
\int_{\Pi} (u-k)^+\omega_r(t-t_0)fdtdx+ \nonumber\\
\int_\Pi H(u-k)[(u-k)f_t+(\varphi(u)-\varphi(k))\cdot\nabla_x f+(g(u)-g(k))\Delta_xf]\theta_r(t-t_0)dtdx\ge 0.
\end{align}

Since
$$
\int_{\Pi} (u-k)^+\omega_r(t-t_0)fdtdx=\int_0^{+\infty}\left(\int_{\R^n} (u(t,x)-k)^+f(t,x)dx\right)\omega_r(t-t_0)dt
$$
and
$t_0$ is a Lebesgue point of the function $t\to\int_{\R^n} (u(t,x)-k)^+f(t,x)dx$, it follows from (\ref{i1}) in the limit as $r\to\infty$ that
\begin{align}\label{i2}
\int_{\R^n} (u(t_0,x)-k)^+f(t_0,x)dx+ \nonumber\\
\int_{(t_0,+\infty)\times\R^n} H(u-k)[(u-k)f_t+(\varphi(u)-\varphi(k))\cdot\nabla_x f+(g(u)-g(k))\Delta_xf]\theta_r(t-t_0)dtdx\ge 0.
\end{align}
Now we pass in (\ref{i2}) to the limit as $E\ni t_0\to 0$.
Since
$$(u(t,x)-k)^+\le (u_0(x)-k)^++(u(t,x)-u_0(x))^+,$$
we obtain that
\begin{align*}
\limsup_{E\ni t_0\to 0}\int_{\R^n}(u(t_0,x)-k)^+f(t_0,x)dx\le\int_{\R^n}(u_0(x)-k)^+f(0,x)dx+\\
\lim_{E\ni t_0\to 0}\int_{\R^n}(u(t_0,x)-u_0(x))^+f(t_0,x)dx=\int_{\R^n}(u_0(x)-k)^+f(0,x)dx,
\end{align*}
where we take into account initial condition (\ref{isub}).
With the help of this relation, the desired inequality (\ref{eint}) follows from (\ref{i2}) in the limit as $E\ni t_0\to 0$.

\medskip
Conversely, assume that relation (\ref{eint}) holds. Taking in this relation a nonnegative test function $f\in C_0^\infty(\Pi)$, we obtain that
$$
\int_\Pi H(u-k)[(u-k)f_t+(\varphi(u)-\varphi(k))\cdot\nabla_x f+(g(u)-g(k))\Delta_xf]dtdx\ge 0.
$$
This means that
$$
((u-k)^+)_t+\div_x[H(u-k)(\varphi(u)-\varphi(k))]-\Delta_x((g(u)-g(k))^+)\le 0 \ \mbox{ in } \D'(\Pi)
$$
and the entropy requirement (\ref{esub}) is satisfied. It only remains to prove initial requirement (\ref{isub}) from Definition~\ref{def1}.
We fix a nonnegative function $h(x)\in C_0^\infty(\R^n)$, and apply (\ref{eint}) to the test function $f=h(x)(1-\theta_r(t-t_0))$, where $t_0\in E$.
As a result, we obtain
\begin{align*}\label{i3}
\int_{\R^n} (u_0(x)-k)^+h(x)dx-\int_{\Pi} (u(t,x)-k)^+\omega_r(t-t_0)hdtdx + \\
\int_{(0,t_0+1/r)\times\R^n} H(u-k)[(\varphi(u)-\varphi(k))\cdot\nabla h+(g(u)-g(k))\Delta h](1-\theta_r(t-t_0))dtdx\ge 0.
\end{align*}
Passing in this relation to the limit as $r\to\infty$, we arrive at the relation
\begin{align*}
\int_{\R^n} (u_0(x)-k)^+h(x)dx-\int_{\R^n} (u(t_0,x)-k)^+h(x)dx + \\
\int_{(0,t_0)\times\R^n} H(u-k)[(\varphi(u)-\varphi(k))\cdot\nabla h+(g(u)-g(k))\Delta h]dtdx\ge 0,
\end{align*}
which implies in the limit as $E\ni t_0\to 0$ that
\begin{equation}\label{i3}
\limsup_{E\ni t_0\to 0} \int_{\R^n}(u(t_0,x)-k)^+h(x)dx\le \int_{\R^n} (u_0(x)-k)^+h(x)dx.
\end{equation}
Obviously, (\ref{i3}) remains valid for all nonnegative $h(x)\in L^1(\R^n)$. We fix $\varepsilon>0$. Since $u_0(x)\in L^\infty(\R^n)$, we can find a step function
$v(x)=\sum_{i=1}^m v_i\chi_{A_i}(x)$, where $v_i\in\R$, $\chi_{A_i}(x)$ are indicator functions of measurable sets $A_i\subset\R^n$, such that $\|u_0-v\|_\infty<\varepsilon$.
The sets $A_i$, $i=1,\ldots,m,$ are supposed to be disjoint. In view of (\ref{i3})
\begin{align}\label{i4}
\limsup_{E\ni t_0\to 0} \int_{\R^n}(u(t_0,x)-v(x))^+h(x)dx=\limsup_{E\ni t_0\to 0} \sum_{i=1}^m \int_{\R^n}(u(t_0,x)-v_i)^+\chi_{A_i}(x)h(x)dx \le \nonumber\\
\sum_{i=1}^m \int_{\R^n}(u_0(x)-v_i)^+\chi_{A_i}(x)h(x)dx= \int_{\R^n} (u_0(x)-v(x))^+h(x)dx\le\varepsilon\|h\|_1.
\end{align}
Since $$(u(t_0,x)-u_0(x))^+\le (u(t_0,x)-v(x))^++(v(x)-u_0(x))^+<(u(t_0,x)-v(x))^++\varepsilon,$$
it follows from (\ref{i4}) that
$$
\limsup_{E\ni t_0\to 0} \int_{\R^n}(u(t_0,x)-u_0(x))^+h(x)dx\le 2\varepsilon\|h\|_1
$$
an in view of arbitrariness of $\varepsilon>0$, we conclude that
$$
\lim_{E\ni t_0\to 0} \int_{\R^n}(u(t_0,x)-u_0(x))^+h(x)dx=0
$$
for all $h(x)\in L^1(\R^n)$. Obviously, this implies that
$$
\esslim_{t\to 0+} (u(t,x)-u_0(x))^+=0 \ \mbox{ in } L^1_{loc}(\R^n)
$$
and completes the proof.
\end{proof}

We will need some a priory estimate of entropy sub-solutions (\textit{maximum principle}).

\begin{proposition}\label{pro2}
If $u=u(t,x)$ is a weak sub-s. of (\ref{1}), (\ref{2}) then $u(t,x)\le b=\esssup u_0(x)$ a.e. in $\Pi$.
\end{proposition}
\begin{proof}
We denote $M=\|u\|_\infty$. For $m>n$ we
integrate (\ref{eint}) over the nonnegative finite measure ${m(m-1)H(M-k)((k-b)^+)^{m-2}dk}$. As a result, we obtain that
for every nonnegative function $f=f(t,x)\in C_0^\infty(\bar\Pi)$
\begin{equation}\label{p1}
\int_\Pi[\eta(u)f_t+\psi(u)\cdot\nabla_x f+h(u)\Delta_x f]dtdx+\int_{\R^n}\eta(u_0(x))f(0,x)dx\ge 0,
\end{equation}
where for $|u|\le M$
\begin{align*}
\eta(u)=m(m-1)\int_b^u (u-k)^+ ((k-b)^+)^{m-2}dk=((u-b)^+)^m, \\
\psi(u)=m(m-1)\int_b^u H(u-k)(\varphi(u)-\varphi(k))^+((k-b)^+)^{m-2}dk\in C(\R,\R^n), \\
h(u)=m(m-1)\int_b^u (g(u)-g(k))^+((k-b)^+)^{m-2}dk\in C(\R).
\end{align*}
We notice that for $|u|\le M$
$$
|\psi(u)|\le 2m(m-1)\max_{|u|\le M}|\varphi(u)|\int_b^u((k-b)^+)^{m-2}dk=2m\max_{|u|\le M}|\varphi(u)|((u-b)^+)^{m-1} \\
$$
(here and in the sequel we denote by $|v|$ the Euclidean norm of a finite-dimensional vector $v$), and analogously
$$
0\le h(u)\le 2m\max_{|u|\le M}|g(u)|((u-b)^+)^{m-1}.
$$
These estimates imply that for each $\varepsilon>0$
$$
\frac{|\psi(u)|}{\eta(u)+\varepsilon}\le \frac{C_1 m}{(u-b)^++\varepsilon((u-b)^+)^{1-m}}, \quad \frac{h(u)}{\eta(u)+\varepsilon}\le \frac{C_2m}{(u-b)^++\varepsilon((u-b)^+)^{1-m}},
$$
where $\displaystyle C_1=2\max_{|u|\le M}|\varphi(u)|$, $\displaystyle C_2=2\max_{|u|\le M}|g(u)|$.
By direct computations we find
$$
\min_{s>0}(s+\varepsilon s^{1-m})=m(m-1)^{\frac{1}{m}-1}\varepsilon^{\frac{1}{m}}.
$$
Therefore,
\begin{equation}\label{p2}
\frac{|\psi(u)|}{\eta(u)+\varepsilon}\le C\varepsilon^{-\frac{1}{m}}, \quad \frac{h(u)}{\eta(u)+\varepsilon}\le C\varepsilon^{-\frac{1}{m}},
\end{equation}
where $C=\max(C_1,C_2)(m-1)^{1-\frac{1}{m}}=\const$.
Remark that $u_0\le b$ a.e. in $\R^n$, consequently $\eta(u_0)=0$ a.e. on $\R^n$, and that $\int_\Pi f_tdtdx=-\int_{\R^n}f(0,x)dx$.
Then it follows from (\ref{p1}) that
\begin{equation}\label{p3}
\int_\Pi[(\eta(u)+\varepsilon)f_t+\psi(u)\cdot\nabla_x f+h(u)\Delta_x f]dtdx+\varepsilon \int_{\R^n}f(0,x)dx\ge 0.
\end{equation}
We choose a nonstrictly decreasing function $\rho(r)\in C^\infty(\R)$ with the properties $\rho(r)=1$ for $r\le 0$, $\rho(r)=e^{-r}$ for $r\ge 1$, it is concave on $(-\infty,1/2]$ and is convex on $[1/2,+\infty)$ (so that $1/2$ is an inflection point of $\rho(r)$). Such a function satisfies the inequality
\begin{equation}\label{ro}
\rho''(r)\le c|\rho'(r)|=-c\rho'(r)
\end{equation}
for some positive constant $c$. In fact, $\rho''(r)\le 0\le |\rho'(r)|$ for $r<1/2$, $\rho''(r)=-\rho'(r)=-e^{-r}$ for $r>1$
while on the segment $[1/2,1]$ we have $-\rho'(r)\ge -\rho'(1)=e^{-1}$ by the convexity of $\rho(r)$, and therefore
$\rho''(r)\le -c\rho'(r)$ where $c=e\max\limits_{1/2\le r\le 1}\rho''(r)\ge 1$. We conclude that (\ref{ro}) holds.
Now we take the test function in the form
$$
f(t,x)=\rho(N(t-t_0)+|x|-1)\theta_r(t_0-t),
$$
where $0<t_0<T$, the constant $N=N(\varepsilon)$ will be indicated later, and the sequence $\theta_r(s)$, $r\in\N$, was defined in the proof of Proposition~\ref{pro1}.
Observe that $f=\theta_r(t_0-t)$ in a neghborhood $|x|<1$ of the set where $x=0$, which implies that $f(t,x)\in C^\infty(\bar\Pi)$.
Since the function $f$ and all its derivatives are exponentially vanishes as $|x|\to\infty$, we may choose the function $f$ as a test function in (\ref{p3}). Observe that
\begin{align}\label{p4a}
f_t(t,x)=N\rho'(N(t-t_0)+|x|-1)\theta_r(t_0-t)-\rho(N(t-t_0)+|x|-1)\omega_r(t_0-t), \\
\label{p4b}
\nabla_x f=\rho'(N(t-t_0)+|x|-1)\theta_r(t_0-t)\frac{x}{|x|}, \\
\label{p5}
\Delta_x f=\left(\rho''(N(t-t_0)+|x|-1)+\rho'(N(t-t_0)+|x|-1)\frac{n-1}{|x|}\right)\theta_r(t_0-t)\le\nonumber\\ -c\rho'(N(t-t_0)+|x|-1)\theta_r(t_0-t)
\end{align}
in view of (\ref{ro}).
It now follows from (\ref{p3}) with the help of (\ref{p4a}), (\ref{p4b} and (\ref{p5}) that for sufficiently large $r\in\N$
\begin{align}\label{p6}
-\int_\Pi[(\eta(u)+\varepsilon)\omega_r(t_0-t)\rho(N(t-t_0)+|x|-1)dtdx+\varepsilon \int_{\R^n}\rho(|x|-Nt_0-1)dx+\nonumber\\
\int_\Pi[N(\eta(u)+\varepsilon)-|\psi(u)|-c h(u)]\rho'(N(t-t_0)+|x|-1)\theta_r(t_0-t)dtdx\ge 0.
\end{align}
Taking in (\ref{p6}) $N=C(1+c)\varepsilon^{-\frac{1}{m}}$, we find that $N(\eta(u)+\varepsilon)-|\psi(u)|-c h(u)\ge 0$
in view of (\ref{p2}). Since $\rho'(r)\le 0$, the last integral in (\ref{p6}) is nonpositive and (\ref{p6}) implies that
$$
\int_\Pi[(\eta(u)+\varepsilon)\omega_r(t_0-t)\rho(N(t-t_0)+|x|-1)dtdx\le \varepsilon \int_{\R^n}\rho(|x|-Nt_0-1)dx.
$$
We assume that $t_0\in E$, where $E\subset\R_+$ is a set of full measure defined in the proof of Proposition~\ref{pro1}.
Then passing to the limit as $r\to\infty$ in the above inequality, we arrive at the relation
\begin{equation}\label{p7}
\int_{\R^n}\eta(u(t_0,x))\rho(|x|-1)dx\le \varepsilon \int_{\R^n}\rho(|x|-Nt_0-1)dx.
\end{equation}
Observe that
\begin{align}\label{p8}
\int_{\R^n}\rho(|x|-Nt_0-1)dx\le \int_{|x|\le Nt_0+2}dx+e^{Nt_0+1}\int_{|x|>Nt_0+2} e^{-|x|}dx\le\nonumber\\
c_n(Nt_0+2)^n+nc_n e^{Nt_0+1}\int_{Nt_0+2}^{+\infty} e^{-r}r^{n-1}dr,
\end{align}
where $c_n$ is the measure of a unit ball in $\R^n$. Since
\begin{align*}
\int_{Nt_0+2}^{+\infty} e^{-r}r^{n-1}dr=\int_0^{+\infty}e^{-s-Nt_0-2}(s+Nt_0+2)^{n-1}ds\le \\ (Nt_0+2)^{n-1}e^{-Nt_0-2}\int_0^{+\infty}e^{-s}(1+s)^{n-1}ds=a(Nt_0+2)^{n-1}e^{-Nt_0-2},
\end{align*}
$a=\const$, it follows from (\ref{p8}) that for some constants $a_1,a_2$
$$
\varepsilon\int_{\R^n}\rho(|x|-N(\varepsilon)t_0-1)dx\le a_1\varepsilon(N(\varepsilon)t_0+2)^n\le a_2\varepsilon(1+\varepsilon^{-\frac{1}{m}})^n\mathop{\to}_{\varepsilon\to 0+}0
$$
(recall that $m>n$). Therefore, passing to the limit in (\ref{p7}) as $\varepsilon\to 0+$, we obtain that for all $t_0\in E$
$$
\int_{\R^n}((u(t_0,x)-b)^+)^m\rho(|x|-1)dx=\int_{\R^n}\eta(u(t_0,x))\rho(|x|-1)dx=0.
$$
Since $\rho(|x|-1)>0$, we conclude that $u(t,x)\le b$ for a.e. $(t,x)\in\Pi$. The proof is complete.
\end{proof}

By Remark~\ref{rem1}(i), we see that any e.super.s $u=u(t,x)$ satisfies the \textit{minimum principle}:
$u(t,x)\ge\essinf u_0(x)$ a.e. in $\Pi$.

\begin{lemma}\label{lem1}
Suppose that $u=u(t,x)\in L^\infty(\Pi)$ is a weak sub-s. of (\ref{1}), (\ref{2}).
Assume also that $\eta(u)\in C^1(\R)$, $\eta'(u)=p(g(u))$, where $p(v)$ is a Lipschitz continuous non-negative nonstrictly increasing function.
Then for each test function $f=f(t,x)\in C_0^\infty(\R)$, $f\ge 0$,
$$
\<\eta(u)_t,f\>=-\int_{\Pi}\eta(u)f_tdtdx\le\int_{\Pi}(\varphi(u)-\nabla_x g(u))\cdot\nabla_x(p(g(u))f)dtdx.
$$
\end{lemma}
\begin{proof}
Since $\eta(u)$ is a convex function, then for each $(t,x)\in\Pi$ and $h>0$
$$
\eta(u(t+h,x))-\eta(u(t,x))\le\eta'(u(t+h,x))(u(t+h,x)-u(t,x))=p(g(u(t+h,x)))(u(t+h,x)-u(t,x)).
$$
Multiplying this inequality by $f(t+h,x)$ and integrating over $(t,x)\in\Pi$, we obtain that for $0<h<\min\{ \ t \ | \ (t,x)\in\supp f \ \}$
\begin{align}\label{l1}
\int_\Pi \eta(u(t,x))(f(t,x)-f(t+h,x))dtdx=\int_\Pi \eta((u(t+h,x))-\eta(u(t,x)))f(t+h,x)dtdx\le \nonumber\\
\int_\Pi p(g(u(t+h,x)))(u(t+h,x)-u(t,x))f(t+h,x)dtdx.
\end{align}
Applying (\ref{wsub}) to a test function $f=a(t)b(x)$ with $a(t)\in C_0^\infty(\R_+)$, $b(x)\in C_0^\infty(\R^n)$, $a(t),b(x)\ge 0$, we obtain that
$$
-\int_0^\infty I(t)a'(t)dt\le \int_\Pi [\varphi(u)-\nabla_x g(u)]\cdot\nabla_x b(x) a(t)dtdx,
$$
where we denote
$\displaystyle I(t)=\int_{\R^n} u(t,x)b(x)dx$. This means that in $\D'(\R_+)$
\begin{equation}\label{l2}
I'(t)\le \int_{\R^n} [\varphi(u(t,x))-\nabla_x g(u(t,x))]\cdot\nabla_x b(x)dx.
\end{equation}
Let $E$ be the set of full measure defined above in the proof of Proposition~\ref{pro1}, and let $t_1,t_2\in E$, $t_2>t_1$. Since $t_1,t_2$ are Lebesgue points of $I(t)$, it follows from (\ref{l2}) that
\begin{equation}\label{l3}
\int_{\R^n}(u(t_2,x)-u(t_1,x))b(x)dx=I(t_2)-I(t_1)\le\int_{(t_1,t_2)\times\R^n}[\varphi(u(t,x))-\nabla_x g(u(t,x))]\cdot\nabla_x b(x)dtdx.
\end{equation}
It is clear that this property remains valid for functions $b(x)$ from the Sobolev space $W_1^1(\R^n)$. In particular, we may take $b=p(g(u(t+h,x)))f(t+h,x)$ for almost all fixed $t$. Then for all such $t$ satisfying the additional requirement
$t,t+h\in E$, we have
\begin{align}\label{l4}
\int_{\R^n}(u(t+h,x)-u(t,x))p(g(u(t+h,x)))f(t+h,x)dx\le \nonumber\\ \int_{(t,t+h)\times\R^n}[\varphi(u(\tau,x))-\nabla_x g(u(\tau,x))]\cdot\nabla_x (p(g(u(t+h,x)))f(t+h,x))d\tau dx.
\end{align}
Putting (\ref{l4}) into (\ref{l1}), we arrive at
\begin{align}\label{l5}
\int_\Pi \eta(u(t,x))(f(t,x)-f(t+h,x))dtdx\le\nonumber\\ \int_{\Pi}\int_t^{t+h}[\varphi(u(\tau,x))-\nabla_x g(u(\tau,x))]\cdot\nabla_x (p(g(u(t+h,x)))f(t+h,x))d\tau dtdx=\nonumber\\ \int_{\Pi}\int_{\tau-h}^{\tau}[\varphi(u(\tau,x))-\nabla_x g(u(\tau,x))]\cdot\nabla_x (p(g(u(t+h,x)))f(t+h,x))dt d\tau dx=\nonumber\\
\int_{\Pi}[\varphi(u(\tau,x))-\nabla_x g(u(\tau,x))]\cdot\nabla_x q_h(\tau,x)d\tau dx,
\end{align}
where we use Fubini theorem and denote
$$
q_h(\tau,x)=\int_{\tau-h}^{\tau} p(g(u(t+h,x)))f(t+h,x)dt=\int_{\tau}^{\tau+h} p(g(u(t,x)))f(t,x)dt.
$$
Observe that
\begin{align}\label{l6}
\frac{1}{h}\nabla_x q_h(\tau,x)=\frac{1}{h}\int_{\tau}^{\tau+h} \nabla_x (p(g(u(t,x)))f(t,x))dt\mathop{\to}_{h\to 0}\nonumber\\
\nabla_x (p(g(u(\tau,x)))f(\tau,x))=p'(g(u(\tau,x)))\nabla_x g(u(\tau,x))f(\tau,x)+p(g(u(\tau,x)))\nabla_x f(\tau,x)
\end{align}
in $L^2_{loc}(\Pi)$ (here we choose the generalized derivative $p'(v)$ being a Borel function). Dividing (\ref{l5}) by $h$ and passing to the limit as $h\to 0$ with the help of (\ref{l6}), we obtain the desired relation
\begin{align*}
-\int_\Pi \eta(u(t,x))f_t(t,x)dtdx\le\int_{\Pi}[\varphi(u(\tau,x))-\nabla_x g(u(\tau,x))]\cdot\nabla_x  (p(g(u(\tau,x)))f(\tau,x))d\tau dx=\\ \int_{\Pi}[\varphi(u(t,x))-\nabla_x g(u(t,x))]\cdot\nabla_x  (p(g(u(t,x)))f(t,x))dtdx.
\end{align*}
\end{proof}

\begin{corollary}\label{cor0}
Let $u=u(t,x)$ be a weak sub-s. of problem (\ref{1}), (\ref{2}), $\|u\|_\infty\le M$. Then for each nonnegative function $\alpha(t)\in C_0^1(\R_+)$
\begin{equation}
\int_\Pi |\nabla_x g(u)|^2e^{-|x|}\alpha(t)dtdx\le C(\alpha,M),
\end{equation}
where $C(\alpha,M)$ is a constant depending only on $\alpha$ and $M$.
\end{corollary}

\begin{proof}
Denote $a=-M$. We apply Lemma~\ref{lem1} with $p(v)=(v-g(a))^+$. Denoting
$\displaystyle\eta(u)=\int_a^u (g(s)-g(a))^+ds$, we obtain the relation
\begin{equation}\label{c1}
\int_{\Pi}\{\eta(u)f_t+(\varphi(u)-\nabla_x g(u))\cdot\nabla_x(p(g(u))f)\}dtdx\ge 0.
\end{equation}
Taking in (\ref{c1}) $f=\alpha(t)e^{-|x|}$ and using the identity $\nabla_x p(g(u))=\nabla_x(g(u)-g(a))=\nabla_x g(u)$, we derive that
\begin{align}\label{c2}
\int_\Pi |\nabla_x g(u)|^2fdtdx\le
\int_\Pi [\eta(u)f_t+\varphi(u)\cdot\nabla_x g(u) f+p(g(u))(\varphi(u)-\nabla_x g(u))\cdot\nabla_x f]dtdx\le\nonumber\\
\int_\Pi [\eta(u)|f_t|+p(g(u))|\varphi(u)|f+(p(g(u))+|\varphi(u)|)|\nabla_xg(u)|f]dtdx,
\end{align}
where we use that $\nabla_x f=-\frac{x}{|x|}f$ and therefore
$$
|(\varphi(u)-\nabla_x g(u))\cdot\nabla_x f|=|(\varphi(u)-\nabla_x g(u))\cdot x/|x||f\le |\varphi(u)-\nabla_x g(u)|f\le
(|\varphi(u)|+|\nabla_x g(u)|)f.
$$
It follows from (\ref{c2}) with the help of Young's inequality that
\begin{align*}
\int_\Pi |\nabla_x g(u)|^2fdtdx\le\int_\Pi [\eta(u)|\alpha'(t)|+p(g(u))|\varphi(u)|\alpha(t)]e^{-|x|}dtdx+ \\ \frac{1}{2}\int_\Pi |\nabla_x g(u)|^2fdtdx+\int_\Pi\frac{1}{2}(p(g(u))+|\varphi(u)|)^2fdtdx,
\end{align*}
which implies that
\begin{align*}
\int_\Pi |\nabla_x g(u)|^2fdtdx\le C(\alpha,M)\doteq \\ \max_{|u|\le M}[2(\eta(u)+p(g(u))|\varphi(u)|)+(p(g(u))+|\varphi(u)|)^2]\int_\Pi\max(\alpha(t),|\alpha'(t)|)e^{-|x|}dtdx,
\end{align*}
as was to be proved.
\end{proof}

Let $H_r(u)=\max(0,\min(1,(1+ru)/2))$, $r\in\N$, be the sequence of approximations of the Heaviside function $H(u)=\sign^+(u)$. Denote by $S=S_g$ the set of $v\in\R$ such that $g^{-1}(v)$ is a singleton.
The following lemma is analogous to \cite[Lemma 5]{Car}.

\begin{lemma}\label{lem2}
Assume that $u=u(t,x)\in L^\infty(\Pi)$ is a weak sub-s. of (\ref{1}), (\ref{2}). Then for each $k\in\R$ such that $g(k)\in S$, and for any test function $f=f(t,x)\in C_0^\infty(\Pi)$, $f\ge 0$,
\begin{align}\label{l7}
\int_\Pi H(u-k)[(u-k)f_t+(\varphi(u)-\varphi(k)-\nabla_x g(u))\cdot\nabla_x f]dtdx\ge \nonumber\\ \limsup_{r\to\infty}
\int_\Pi H_r'(g(u)-g(k))|\nabla_x g(u)|^2f dtdx\ge 0.
\end{align}
\end{lemma}

\begin{proof}
Since $g(k)\in S$, then $$H(u-k)=H(g(u)-g(k))=\lim_{r\to\infty} H_r(g(u)-g(k))$$
whenever $g(u)\not=g(k)$ (notice that $H_r(0)=1/2$ for all $r\in\N$).
Let $\eta_r(u)=\int_k^u H_r(g(s)-g(k))ds$. Obviously, $\eta_r(u)\to (u-k)^+$ as $r\to\infty$ uniformly in $u$. Applying Lemma~\ref{lem1} with $p(v)=H_r(v-g(k))$, we obtain that for each $f=f(t,x)\in C_0^\infty(\Pi)$, $f\ge 0$
\begin{align}\label{l8}
\int_\Pi \{\eta_r(u)f_t+((\varphi(u)-\varphi(k))-\nabla_x g(u))\cdot\nabla_x (H_r(g(u)-g(k))f)\}dtdx=\nonumber\\
\int_\Pi \{\eta_r(u)f_t+(\varphi(u)-\nabla_x g(u))\cdot\nabla_x (H_r(g(u)-g(k))f)\}dtdx\ge 0,
\end{align}
where we take into account that the vector $\int_\Pi \nabla_x (H_r(g(u)-g(k))f)dtdx=0$.
Since $$\nabla_x (H_r(g(u)-g(k))f)=fH_r'(g(u)-g(k))\nabla_x g(u)+H_r(g(u)-g(k))\nabla_x f,$$ (\ref{l8}) implies that
\begin{align}\label{l9}
\int_\Pi \{\eta_r(u)f_t+H_r(g(u)-g(k))((\varphi(u)-\varphi(k))-\nabla_x g(u))\cdot\nabla_xf\}dtdx+\nonumber\\
\int_\Pi fH_r'(g(u)-g(k))(\varphi(u)-\varphi(k))\cdot\nabla_x g(u)dtdx-\int_\Pi fH_r'(g(u)-g(k))|\nabla_x g(u)|^2dtdx\ge 0.
\end{align}
Now, we are going to pass in (\ref{l9}) to the limit as $r\to\infty$. Since $\nabla_x g(u)=0$ a.e. on the set where $g(u)=g(k)$, it is clear that the first integral
\begin{align}\label{l10}
\int_\Pi \{\eta_r(u)f_t+H_r(g(u)-g(k))((\varphi(u)-\varphi(k))-\nabla_x g(u))\cdot\nabla_xf\}dtdx\mathop{\to}_{r\to\infty}
\nonumber\\ \int_\Pi H(u-k)[(u-k)f_t+(\varphi(u)-\varphi(k)-\nabla_x g(u))\cdot\nabla_x f]dtdx.
\end{align}
The second integral can be treated in the same way as in the proof of \cite[Lemma 1]{Car}. Let $g^{-1}_0(v)$, where $v\in g(\R)$, be a point in $g^{-1}(v)$ of minimal absolute value. Obviously, $u=g^{-1}_0(g(u))$ whenever $g(u)\in S$ while
$\nabla_x g(u(t,x))$ almost everywhere on the set of $(t,x)$ where $g(u)\notin S$. Therefore
\begin{align}\label{l11}
I_r=\int_\Pi fH_r'(g(u)-g(k))(\varphi(u)-\varphi(k))\cdot\nabla_x g(u)dtdx=\nonumber\\ \int_\Pi fH_r'(g(u)-g(k))(\varphi(g^{-1}_0(g(u)))-\varphi(k))\cdot\nabla_x g(u)dtdx=\int_\Pi \div_x F_r(g(u)) f dtdx,
\end{align}
where we denote
$$
F_r(v)=\int_{g(k)}^v H_r'(s-g(k))(\varphi(g^{-1}_0(s))-\varphi(k))ds=\frac{r}{2}\int_{g(k)}^{v'}(\varphi(g^{-1}_0(s))-\varphi(k))ds,
$$
where $v'=\max(g(k)-1/r,\min(g(k)+1/r,v))$. It is clear that
$$
|F_r(v)|\le \frac{r}{2}\int_{g(k)-1/r}^{g(k)+1/r}|\varphi(g^{-1}_0(s))-\varphi(k)|ds\mathop{\to}_{r\to\infty} 0
$$
since the vector function $\varphi(g^{-1}_0(s))$ is continuous at the point $g(k)\in S$ and $\varphi(g^{-1}_0(g(k)))=\varphi(k)$. By Lebesgue dominated convergence theorem we deduce from (\ref{l11}) that
\begin{equation}\label{l12}
I_r=-\int_\Pi F_r(g(u))\cdot\nabla_x f dtdx\to 0 \ \mbox{ as } r\to\infty.
\end{equation}
Taking into account (\ref{l10}), (\ref{l12}), we deduce from (\ref{l9}) in the limit as $r\to\infty$ the desired relation
(\ref{l7}).
\end{proof}

\begin{corollary}\label{cor1} Assume that the function $g(u)$ is strictly increasing and $u=u(t,x)$ is a weak sub-s. (weak super-s.) of (\ref{1}), (\ref{2}). Then $u$ is an e.sub-s. (e.super-s.) of this problem.
\end{corollary}

\begin{proof}
If $u$ is a weak sub-s. of (\ref{1}), (\ref{2}) then it satisfies the assumptions of Lemma~\ref{lem2}. Further, since $g(u)$ is strictly increasing then $g(k)\in S$ for every $k\in\R$. In view of (\ref{l7}) the function $u$ satisfies entropy relation (\ref{esub}). Hence, it is an e.sub-s. Now, suppose that $u$ is a weak super-s. of (\ref{1}), (\ref{2}). In view of relation (\ref{wsuper})
$$
(-u)_t+\div_x(-\varphi(u))-\Delta_x(-g(u))=-[u_t+\div_x \varphi(u)-\Delta_x g(u)]\le 0 \ \mbox{ in } \D'(\Pi),
$$
that is, the function $-u$ satisfies relation (\ref{wsub}) subject to equation (\ref{red}). Thus, the function $-u$ is a weak sub-s. of the problem (\ref{red}). As was already proved, $-u$ is an e.sub-s. of this problem. But this means that the function $u$ is an e.super-s. of original problem (\ref{1}), (\ref{2}), see Remark~\ref{rem1}(i).
\end{proof}

%

\section{Main results}

\begin{theorem}\label{th1}
Let $u_1=u_1(t,x)$, $u_2=u_2(t,x)$ be e.sub-s. of (\ref{1}), (\ref{2}). Then $u=\max(u_1,u_2)$ is an e.sub-s. of this problem as well.
\end{theorem}

\begin{proof}
By the definition $\nabla_x g(u_1),\nabla_x g(u_2)\in L^2_{loc}(\Pi)$. Since $g(u)=\max(g(u_1),g(u_2))=(g(u_1)-g(u_2))^++g(u_2)$, then
\begin{align*}
\nabla_x g(u)=H(g(u_1)-g(u_2))(\nabla_x g(u_1)-\nabla_x g(u_2))+\nabla_x g(u_2)=\\
H(u_1-u_2)(\nabla_x g(u_1)-\nabla_x g(u_2))+\nabla_x g(u_2)\in L^2_{loc}(\Pi),
\end{align*}
where we take into account the fact that
$\nabla_x g(u_1)-\nabla_x g(u_2)=\nabla_x(g(u_1)-g(u_2))=0$ almost everywhere on the set $\{(t,x) | g(u_1)=g(u_2)\}$.
Further, $$(u-u_0)^+=\max((u_1-u_0)^+,(u_2-u_0)^+)\le (u_1-u_0)^++(u_2-u_0)^+.$$
Therefore, it follows from the initial relations (\ref{isub}) for e.sub-s. $u_1,u_2$ that
$$
\esslim_{t\to 0} (u(t,x)-u_0(x))^+=0 \ \mbox{ in } L^1_{loc}(\R^n).
$$
It only remains to verify the entropy relation (\ref{esub1}). For that we apply the doubling variables technique.
Namely, we consider the function $u_2$ as a function of new variables $(s,y)\in\Pi$. Taking in (\ref{esub1})
$k=u_2(s,y)$, we obtain that
$$
(\max(u_1,u_2))_t+\div_x\varphi(\max(u_1,u_2))-\Delta_x g(\max(u_1,u_2))\le 0  \ \mbox{ in } \D'(\Pi).
$$
Therefore, for each nonnegative test function $f=f(t,x;s,y)\in C_0^\infty(\Pi\times\Pi)$ and all $(s,y)\in\Pi$
\begin{equation}\label{3}
\int_{\Pi}\{\max(u_1,u_2))f_t+[\varphi(\max(u_1,u_2))-H(u_1-u_2)\nabla_x g(u)]\cdot\nabla_x f\}dtdx\ge 0.
\end{equation}
Moreover, if $(s,y)\in D_2\doteq\{ \ (s,y)\in\Pi \ | \ g(u_2(s,y))\in S_g \ \}$, then by Lemma~\ref{lem2}
\begin{align}\label{3a}
\int_{\Pi}\{\max(u_1,u_2))f_t+[\varphi(\max(u_1,u_2))-H(u_1-u_2)\nabla_x g(u_1)]\cdot\nabla_x f\}dtdx\ge \nonumber\\
\limsup_{r\to\infty}
\int_\Pi H_r'(g(u_1)-g(u_2))|\nabla_x g(u_1)|^2f dtdx.
\end{align}
It follows from (\ref{3}), (\ref{3a}), after integration with respect to $(s,y)$, that
\begin{align}\label{4}
\int_{\Pi\times\Pi}\{\max(u_1,u_2))f_t+[\varphi(\max(u_1,u_2))-H(u_1-u_2)\nabla_x g(u_1)]\cdot\nabla_x f\}dtdxdsdy\ge \nonumber\\
\limsup_{r\to\infty}
\int_{\Pi\times D_2} H_r'(g(u_1)-g(u_2))|\nabla_x g(u_1)|^2f dtdxdsdy=\nonumber\\ \limsup_{r\to\infty}
\int_{D_1\times D_2} H_r'(g(u_1)-g(u_2))|\nabla_x g(u_1)|^2f dtdxdsdy,
\end{align}
where we denote $D_1=\{ \ (t,x)\in\Pi \ | \ g(u_1(t,x))\in S_g \ \}$. In (\ref{4}) we take into account that $\nabla_x g(u_1)=0$ a.e. on the complement of the set $D_1$.

Similarly, changing places of variables $(t,x)$ and $(s,y)$ and taking into account that $u_2=u_2(s,y)$ is an e.sub-s. of the equation $u_s+\div_y\varphi(u))-\Delta_y(g(u))=0$, we obtain the inequality
\begin{align}\label{4a}
\int_{\Pi\times\Pi}\{\max(u_1,u_2))f_s+[\varphi(\max(u_1,u_2))-H(u_2-u_1)\nabla_y g(u_2)]\cdot\nabla_y f\}dtdxdsdy\ge \nonumber\\
\limsup_{r\to\infty}
\int_{D_1\times D_2} H_r'(g(u_2)-g(u_1))|\nabla_y g(u_2)|^2f dtdxdsdy.
\end{align}
Since, evidently, for each $r\in\N$
\begin{align*}
0=\int_{\Pi\times\Pi}\nabla_x g(u_1)\cdot\nabla_y (H_r(g(u_1)-g(u_2))f)dtdxdsdy= \\
-\int_{\Pi\times\Pi}H_r'(g(u_1)-g(u_2))\nabla_x g(u_1)\cdot\nabla_y g(u_2)fdtdxdsdy+\\
\int_{\Pi\times\Pi}H_r(g(u_1)-g(u_2))\nabla_x g(u_1)\cdot\nabla_y fdtdxdsdy, \\
0=\int_{\Pi\times\Pi}\nabla_y g(u_2)\cdot\nabla_x (H_r(g(u_2)-g(u_1))f)dtdxdsdy= \\
-\int_{\Pi\times\Pi}H_r'(g(u_2)-g(u_1))\nabla_x g(u_1)\cdot\nabla_y g(u_2)fdtdxdsdy+\\
\int_{\Pi\times\Pi}H_r(g(u_2)-g(u_1))\nabla_y g(u_2)\cdot\nabla_x fdtdxdsdy,
\end{align*}
we arrive at the following limit relations
\begin{align}
\label{5a}
-\int_{\Pi\times\Pi}H(u_1-u_2)\nabla_x g(u_1)\cdot\nabla_y fdtdxdsdy=
-\int_{\Pi\times\Pi}H(g(u_1)-g(u_2))\nabla_x g(u_1)\cdot\nabla_y fdtdxdsdy=\nonumber\\ -\lim_{r\to\infty}
\int_{\Pi\times\Pi}H_r'(g(u_1)-g(u_2))\nabla_x g(u_1)\cdot\nabla_y g(u_2)fdtdxdsdy=\nonumber\\
-\lim_{r\to\infty}
\int_{D_1\times D_2}H_r'(g(u_1)-g(u_2))\nabla_x g(u_1)\cdot\nabla_y g(u_2)fdtdxdsdy; \\
\label{5b}
-\int_{\Pi\times\Pi}H(u_2-u_1)\nabla_y g(u_2)\cdot\nabla_x fdtdxdsdy=-\int_{\Pi\times\Pi}H(g(u_2)-g(u_1))\nabla_y g(u_2)\cdot\nabla_x fdtdxdsdy=\nonumber\\
-\lim_{r\to\infty}\int_{D_1\times D_2}H_r'(g(u_2)-g(u_1))\nabla_x g(u_1)\cdot\nabla_y g(u_2)fdtdxdsdy,
\end{align}
where we use that $H_r'(s)\mathop{\to}\limits_{r\to\infty} H(s)$ for $s\not=0$ while $\nabla_x g(u_1)=0$ a.e. on the set, where $g(u_1)-g(u_2)=0$.

Putting relations (\ref{4}), (\ref{4a}), (\ref{5a}), (\ref{5b}) together, we obtain that
\begin{align}\label{6}
\int_{\Pi\times\Pi}\{\max(u_1,u_2))(f_t+f_s)+\varphi(\max(u_1,u_2))\cdot(\nabla_x+\nabla_y) f- \nonumber\\ (H(u_1-u_2)\nabla_x g(u_1)+H(u_2-u_1)\nabla_y g(u_2))\cdot(\nabla_x+\nabla_y) f\}dtdxdsdy\ge \nonumber\\
\limsup_{r\to\infty}
\int_{D_1\times D_2} H_r'(g(u_1)-g(u_2))|\nabla_x g(u_1)-\nabla_y g(y_2)|^2f dtdxdsdy\ge 0,
\end{align}
where we use the fact that $H_r'(s)$ is an even function.
Since
$$
H(u_1-u_2)\nabla_x g(u_1)+H(u_2-u_1)\nabla_y g(u_2)=(\nabla_x+\nabla_y)g(\max(u_1,u_2)),
$$
we can rewrite (\ref{6}) in the form
\begin{align}\label{7}
\int_{\Pi\times\Pi}\{\max(u_1,u_2)(f_t+f_s)+\varphi(\max(u_1,u_2))\cdot(\nabla_x+\nabla_y) f+ \nonumber\\ g(\max(u_1,u_2))(\nabla_x+\nabla_y)\cdot(\nabla_x+\nabla_y) f\}dtdxdsdy\ge 0.
\end{align}
Let $\delta_r(t,x)=\omega_r(t)\prod_{i=1}^n\omega_r(x_i)$, where $t\in\R$, $x=(x_1,\ldots,x_n)\in\R^n$, and the sequence
$\omega_r(s)$, $r\in\N$, was introduced above in the proof of Proposition~\ref{pro1}. We apply (\ref{7}) to the test function $f=h(t,x)\delta_r(t-s,x-y)$, where $h=h(t,x)\in C_0^\infty(\Pi)$, $h\ge 0$. It is clear that $f\in C_0^\infty(\Pi\times\Pi)$ for sufficiently large $r$, $f\ge 0$. Since
$$
(\partial_t+\partial_s)\delta_r(t-s,x-y)=0, \quad (\nabla_x+\nabla_y)\delta_r(t-s,x-y)=0,
$$
it follows from (\ref{7}) that
\begin{equation}\label{8}
\int_{\Pi\times\Pi}\{\max(u_1,u_2)h_t+\varphi(\max(u_1,u_2))\cdot\nabla_x h+ g(\max(u_1,u_2))\Delta_x h\}\delta_r(t-s,x-y)dtdxdsdy\ge 0.
\end{equation}
Observe that
\begin{align*}
|\max(u_1(t,x),u_2(s,y))-\max(u_1(t,x),u_2(t,x))|\le |u_2(s,y)-u_2(t,x)|, \\ |\varphi(\max(u_1(t,x),u_2(s,y)))-\varphi(\max(u_1(t,x),u_2(s,y)))|\le\mu_\varphi(|u_2(s,y)-u_2(t,x)|), \\
|g(\max(u_1(t,x),u_2(s,y)))-g(\max(u_1(t,x),u_2(s,y)))|\le\mu_g(|u_2(s,y)-u_2(t,x)|),
\end{align*}
where
\begin{align*}
\mu_\varphi(\sigma)=\max\{ \ |\varphi(u)-\varphi(v)| \ | \ u,v\in [-M,M], |u-v|\le\sigma \ \}, \\
\mu_g(\sigma)=\max\{ \ |g(u)-g(v)| \ | \ u,v\in [-M,M], |u-v|\le\sigma \
\end{align*}
are continuity modules of the vector function $\varphi(u)$ and the function $g(u)$, respectively, on the segment $[-M,M]$ with $M=\|u_2\|_\infty$.
It follows from these estimates that
\begin{align}\label{9}
\int_\Pi\{\max(u_1,u_2)h_t+\varphi(\max(u_1,u_2))\cdot\nabla_x h+ g(\max(u_1,u_2))\Delta_x h\}\delta_r(t-s,x-y)dsdy\mathop{\to}_{r\to\infty} \nonumber\\
\max(u_1(t,x),u_2(t,x))h_t(t,x)+\varphi(\max(u_1(t,x),u_2(t,x)))\cdot\nabla_x h(t,x)+\nonumber\\ g(\max(u_1(t,x),u_2(t,x)))\Delta_x h(t,x)
\end{align}
for each $(t,x)$ from the set of full measure of Lebesgue points of the function $u_2$. By Lebesgue dominated convergence theorem, we derive from (\ref{9}) the limit relation
\begin{align*}
\int_{\Pi\times\Pi}\{\max(u_1,u_2)h_t+\varphi(\max(u_1,u_2))\cdot\nabla_x h+ g(\max(u_1,u_2))\Delta_x h\}\delta_r(t-s,x-y)dtdxdsdy \nonumber\\ \mathop{\to}_{r\to\infty}
\int_\Pi\{\max(u_1,u_2)h_t+\varphi(\max(u_1,u_2))\cdot\nabla_x h+ g(\max(u_1,u_2))\Delta_x h\}dtdx
\end{align*}
(in the left integral $u_2=u_2(s,y)$ while in the right integral $u_2=u_2(t,x)$).
In view of (\ref{8}), this relation implies that
$$
\int_\Pi\{\max(u_1,u_2)h_t+\varphi(\max(u_1,u_2))\cdot\nabla_x h+ g(\max(u_1,u_2))\Delta_x h\}dtdx\ge 0
$$
for all nonnegative test function $h\in C_0^\infty(\Pi)$, that is,
\begin{equation}\label{10}
(\max(u_1,u_2))_t+\div_x\varphi(\max(u_1,u_2))-\Delta_x g(\max(u_1,u_2))\le 0 \ \mbox{ in } \D'(\Pi).
\end{equation}
As directly follows from Remark~\ref{rem1}(iii), for every $k\in\R$ the functions $\max(u_1,k)$, $\max(u_2,k)$ are e.sub-s. of (\ref{1}), (\ref{2}) with initial data
$\max(u_0(x),k)$. Placing these e.sub-s. in (\ref{10}) instead of $u_1,u_2$, we obtain that for all $k\in\R$
$$
(\max(u_1,u_2,k))_t+\div_x\varphi(\max(u_1,u_2,k))-\Delta_x g(\max(u_1,u_2,k))\le 0 \ \mbox{ in } \D'(\Pi).
$$
This means that $u=\max(u_1,u_2)$ satisfies entropy relation (\ref{esub1}). The proof is complete.
\end{proof}
The following extension of Theorem~\ref{th1} is proved by induction in the number of functions $m$.
\begin{corollary}\label{cor3}
If $u_i=u_i(t,x)$, $i=1,\dots,m$, are e.sub-s. of (\ref{1}), (\ref{2}). Then $u=\max\limits_{i=1,\ldots,m} u_i(t,x)$ is an e.sub-s. of (\ref{1}), (\ref{2}) as well.
\end{corollary}

In view of Remark~\ref{rem1}(i), minimum of a finite family of e.super-s. of (\ref{1}), (\ref{2}) is an e.super-s. of the same problem.

Now, we are ready to establish existence of the largest e.sub-s. and the smallest e.super-s. of our problem.

\subsection{Proof of Theorem~\ref{th2}}
We denote by $Sub$ the set of e.sub-s. of the problem (\ref{1}), (\ref{2}). This set is not empty since it includes an existing e.s. By Proposition~\ref{pro2} $u\le b=\esssup u_0(x)$ for every $u\in Sub$ and therefore
$$
I(u)\doteq\int_\Pi u(t,x)e^{-t-|x|}dtdx\le b\int_\Pi e^{-t-|x|}dtdx=\const.
$$
This implies that there exists a finite value $I=\sup\limits_{u\in Sub} I(u)$. We choose the sequence $u_r$, $r\in\N$,
such that $I(u_r)\ge I-1/r$. By Corollary~\ref{cor3} the functions $v_r=\max\limits_{i=1,r} u_i(t,x)$ are e.sub-s. of the problem (\ref{1}), (\ref{2}). Since $v_r\ge u_r$, then $1-1/r\le I(v_r)\le I$. Obviously, the sequence $v_r$ is increasing and bounded, $a\doteq\essinf u_1\le v_r(t,x)\le b$ for a.e. $(t,x)\in\Pi$. Therefore, the sequence $v_r$ converges as $r\to\infty$ to some function $u_+(t,x)$ a.e. on $\Pi$ and in $L^1_{loc}(\Pi)$ as well. It is clear that $u_+(t,x)\in L^\infty(\Pi)$, $a\le u_+\le b$ a.e. on $\Pi$. Let us show that $u_+$ is an e.sub-s. of (\ref{1}), (\ref{2}).
Observe that $\nabla_x g(v_r)\in L^2_{loc}(\Pi,\R^n)$ and in view of Corollary~\ref{cor0} (with $M=\max(|a|,|b|)$), the sequence $\nabla_x g(v_r)$ is bounded in $L^2_{loc}(\Pi,\R^n)$. Passing to a subsequence if necessary, we may assume that this sequence converges as $r\to\infty$ to some vector $p=p(t,x)$ weakly in $L^2_{loc}(\Pi,\R^n)$. From the identity
$$
\int_\Pi g(v_r)\nabla_x f dtdx=-\int_\Pi f\nabla_x g(v_r) dtdx, \quad f=f(t,x)\in C_0^\infty(\Pi)
$$
it follows, in the limit as $r\to\infty$, that
$$
\int_\Pi g(u_+)\nabla_x f dtdx=-\int_\Pi fp dtdx, \quad \forall f=f(t,x)\in C_0^\infty(\Pi).
$$
This means that $\nabla_x g(u_+)=p\in L^2_{loc}(\Pi,\R^n)$ in $\D'(\Pi)$. Further, e.sub-s. $v_r$ satisfy (\ref{eint})
$$
\int_\Pi [(v_r-k)^+f_t+H(v_r-k)(\varphi(v_r)-\varphi(k))\cdot\nabla_x f+(g(v_r)-g(k))^+\Delta_x f]dtdx+\int_{\R^n} (u_0(x)-k)^+f(0,x)dx\ge 0
$$
for each $k\in\R$ and any nonnegative test function $f=f(t,x)\in C_0^\infty(\bar\Pi)$. In the limit as $r\to\infty$ we obtain that the limit function $u_+$ satisfies (\ref{eint}) as well:
$$
\int_\Pi [(u_+-k)^+f_t+H(u_+-k)(\varphi(u_+)-\varphi(k))\cdot\nabla_x f+(g(u_+)-g(k))^+\Delta_x f]dtdx+\int_{\R^n} (u_0(x)-k)^+f(0,x)dx\ge 0.
$$
By Proposition~\ref{pro1} we see that $u_+$ is an e.sub-s. of (\ref{1}), (\ref{2}). Notice also that
$$
I(u_+)=\lim_{r\to\infty} I(v_r)=I.
$$
Now we demonstrate that $u_+$ is the largest e.sub-s. In fact, if $u\in Sub$ then $v=\max(u_+,u)\in Sub$ as well,
by Theorem~\ref{th1}.
Therefore,
$$
\int_\Pi (v(t,x)-u_+(t,x))e^{-t-|x|}dtdx=I(v)-I\le 0.
$$
Since $v\ge u_+$, this implies that $v=u_+$ a.e. on $\Pi$, which is equivalent to the relation $u\le u_+$ a.e. on $\Pi$.
Thus $u_+$ is larger than any e.sub-s. and therefore it is the largest e.sub-s. of the problem (\ref{1}), (\ref{2}).
Clearly such an e.sub-s. is unique.

As we already have proved, there exists $v_+=v_+(t,x)$ the largest e.sub-s. of problem (\ref{red}). Then by Remark~\ref{rem1}(i) the function
$u_-(t,x)=-v_+(t,x)$ is the smallest e.super-s. of the original problem.

Finally, let $u=u(t,x)$ be an e.s. of (\ref{1}), (\ref{2}). Since $u$ is an e.sub-s. and e.super-s. of this problem, we conclude that $u_-\le u\le u_+$ a.e. on $\Pi$. This completes the proof.

\subsection{The case of periodic initial data}
Now we assume that the initial function $u_0(x)$ is periodic. Without loss of generality, we may assume that the lattice of periods is the standard lattice $\Z^n$. Hence, $u_0(x+e)=u_0(x)$ a.e. on $\R^n$ for each $e\in\Z^n$.

\begin{theorem}\label{th3}
The largest e.sub-s. $u_+$ and the smallest e.super-s. $u_-$ of the problem (\ref{1}), (\ref{2}) are space-periodic and coincide: $u_+=u_-$.
\end{theorem}

\begin{proof}
Let $e\in\Z^m$. In view of periodicity of the initial function it is obvious that $u(t,x+e)$ is an e.sub-s. of (\ref{1}), (\ref{2}) if and only if $u(t,x)$ is an e.sub-s. of the same problem. Therefore, $u_+(t,x+e)$ is the largest e.sub-s.
of (\ref{1}), (\ref{2}) together with $u_+$. By the uniqueness $u_+(t,x+e)=u_+(t,x)$ a.e. on $\Pi$ for all $e\in\Z^n$, that is $u_+$ is a space periodic function. In the same way we prove space periodicity of the minimal e.super-s. $u_-$.
In view of (\ref{wsub}), (\ref{wsuper}) we have
$$
(u_+)_t+\div_x\varphi(u_+)-\Delta_x g(u_+)\le 0, \
(u_-)_t+\div_x\varphi(u_-)-\Delta_x g(u_-)\ge 0 \ \mbox{ in } \D'(\Pi).
$$
Subtracting the second inequality from the first one, we obtain the relation
\begin{equation}\label{11}
(u_+-u_-)_t+\div_x(\varphi(u_+)-\varphi(u_-))-\Delta_x (g(u_+)-g(u_-))\le 0 \ \mbox{ in } \D'(\Pi).
\end{equation}
Let $\alpha(t)\in C_0^1(\R_+)$, $\beta(y)\in C_0^2(\R^n)$, $\alpha(t),\beta(y)\ge 0$, $\displaystyle\int_{\R^n}\beta(y)dy=1$. Applying (\ref{11}) to the test function $\alpha(t)\beta(x/k)$, with $k\in\N$, we arrive at the relation
\begin{align*}
\int_\Pi(u_+-u_-)\alpha'(t)\beta(x/k)dtdx+k^{-1}\int_\Pi(\varphi(u_+)-\varphi(u_-))\cdot\nabla_y\beta(x/k)\alpha(t)dtdx+\\
k^{-2}\int_\Pi(g(u_+)-g(u_-))\Delta_y\beta(x/k)\alpha(t)dtdx\ge 0.
\end{align*}
Multiplying this inequality by $k^{-n}$ and passing to the limit as $k\to\infty$, we obtain
\begin{equation}\label{12}
\int_{\R_+\times P}(u_+(t,x)-u_-(t,x))\alpha'(t)dtdx\ge 0,
\end{equation}
where $P=[0,1)^n$ is the periodicity cell. We use here the known property
$$
\lim_{k\to\infty}k^{-n} \int_\Pi\mu(t,x)\alpha(t)\beta(x/k)dtdx=\int_{\R_+\times P}\alpha(t)\mu(t,x)dtdx
$$
for an arbitrary $x$-periodic function $\mu(t,x)\in L^1_{loc}(\Pi)$. Identity (\ref{12}) means that
$$
\frac{d}{dt}\int_P(u_+(t,x)-u_-(t,x))dx\le 0 \ \mbox{ in } \D'(\R_+).
$$
This implies that for a.e. $t,t_0$, $t>t_0$
\begin{equation}\label{13}
\int_P(u_+(t,x)-u_-(t,x))dx\le \int_P(u_+(t_0,x)-u_-(t_0,x))dx.
\end{equation}
Taking into account the initial relations (\ref{isub}), (\ref{isuper}), we find that
$$
\int_P(u_+(t_0,x)-u_-(t_0,x))dx\le\int_P(u_+(t_0,x)-u_0(x))^+dx+\int_P(u_0(x)-u_-(t_0,x))^+dx\to 0
$$
as $t_0\to 0$ running over a set of full measure. Therefore (\ref{13}) implies in the limit as $t_0\to 0$ that
$$\int_P(u_+(t,x)-u_-(t,x))dx=0$$ for a.e. $t>0$. Since $u_+\ge u_-$, we conclude that $u_+=u_-$ a.e. on $\Pi$.
\end{proof}
Since any e.s. of (\ref{1}), (\ref{2}) is situated between $u_-$ and $u_+$, we deduce the following

\begin{corollary}\label{cor4}
An e.s. of (\ref{1}), (\ref{2}) is unique and coincides with $u_+$.
\end{corollary}

More generally we establish below the comparison principle formulated in Theorem~\ref{th4}.
\subsection{Proof of Theorem~\ref{th4}}

For definiteness suppose that the function $u_0(x)$ is periodic. The case of periodic $v_0$ is treated similarly. By Theorem~\ref{th3}  the  functions $u_+=u_-$ coincide with the unique e.s. of (\ref{1}), (\ref{2}). Since $u_+$ is the largest e.sub-s of this problem, $u\le u_+=u_-$. It is clear that
the function $\min(u_-,v)$ is an e.super-s. of (\ref{1}), (\ref{2}) with initial function $u_0$ and since $u_-$ is the smallest e.super-s. of this problem, we conclude that $u\le u_-\le \min(u_-,v)\le v$, as was to be proved.

\section{Conclusion}
We underline that for conservation laws (\ref{con}) Theorems~\ref{th2},~\ref{th4} were establishes in \cite{PaMax1,PaMax2,PaIzv}.
Moreover, it was demonstrated in \cite{PaMax2,PaIzv} that the functions $u_\pm$ are actually e.s. of (\ref{con}), (\ref{2}). The same result can be proved in the parabolic case as well. The comparison principle and the uniqueness of e.s. remain valid in the case when the initial function is periodic at least in $n-1$ independent directions, this can be proved by the same methods as for conservations laws, see \cite{PaMax2,PaIzv}.

\section*{Acknowledgements}
This work was supported by the Ministry of Science and Higher Education of the Russian  Federation (project no. 1.445.2016/1.4) and by the Russian Foundation for Basic Research (grant 18-01-00258-a.)

\end{document}